\documentclass[11pt]{article}
\usepackage{amsmath,amssymb,amsthm,graphicx,fixmath}

\usepackage{algpseudocode}
\usepackage{algorithm}
\usepackage{setspace}
\usepackage{ifthen}
\usepackage{tikz}

\renewcommand{\paragraph}[1]{\par\medskip\noindent\textbf{#1}\ }

\def\Re{\mathbb R}

\providecommand{\remove}[1]{}

\theoremstyle{plain}
\newtheorem{theorem}{Theorem}[section]

\newtheorem{proposition}[theorem]{Proposition}
\newtheorem{claim}[theorem]{Claim}

\newtheorem{observation}[theorem]{Observation}
\theoremstyle{definition}
\newtheorem{definition}[theorem]{Definition}
\theoremstyle{remark}

\newcommand{\A}{\mathcal{A}}

\newcommand{\D}{\mathcal{D}}

\renewcommand{\ddagger}{{\star}}

\makeatletter
\renewcommand*\@fnsymbol[1]{\ensuremath{%
  \ifcase#1\or
    *\or
    \star\or     % <-- זה הסימון של ה-thanks השני (במקום †)
    \ddagger\or
    \mathsection\or
    \mathparagraph\or
    \|\or
    **\or
    \star\star\or
    \ddagger\ddagger
  \else
    \@ctrerr
  \fi}}
\makeatother

\begin{document}
	
	\title{On the Largest Convexity Number of Co-Finite Sets in the Plane}
\author{Chaya Keller\thanks{School of Computer Science, Ariel University, Israel. \texttt{chayak@ariel.ac.il}. Research partially supported by the Israel Science Foundation (grant no. 1065/20).}
	\mbox{ }
	and Micha A. Perles\thanks{Einstein Institute of Mathematics, Hebrew University, Jerusalem, Israel.
		\texttt{micha.perles@mail.huji.ac.il}}
}

	\maketitle

	\begin{abstract}
		The convexity number of a set $X \subset \mathbb{R}^2$ is the minimum number of convex subsets required to cover it. We study the following question: what is the largest possible convexity number $f(n)$ of $\mathbb{R}^2 \setminus S$, where $S$ is a set of $n$ points in general position in the plane?
        We prove that for all $n \geq 4$, $\lfloor\frac{n+5}{2}\rfloor \leq f(n) \leq \frac{7n+44}{11}$. We also show that for every $n \geq 4$, if the points of $S$ are in convex position then the  convexity number of $\Re^2 \setminus S$ is $\lfloor\frac{n+5}{2}\rfloor$. This solves a problem of Lawrence and Morris [Finite sets as complements of finite unions of convex sets, Disc. Comput. Geom. 42 (2009), 206-218].
	\end{abstract}

    \section{Introduction}

The \emph{convexity number} of a set $X \subset \mathbb{R}^d$, denoted by $\gamma(X)$, is the minimum number of convex sets required to cover it. In 1957, Valentine~\cite{Valentine57} showed that if among any three points in a closed set $X \subset \mathbb{R}^2$ there are two points that see each other through $X$ (i.e., the interval connecting them is included in $X$), then $\gamma(X) \leq 3$. Following this work, numerous papers (e.g.,~\cite{BreenK76,Eggleston74,NitzanP13,PerlesS90}) studied the question of bounding the convexity number of a set $X \subset \mathbb{R}^2$ in terms of its \emph{invisibility number} $\omega(X)$, defined as the largest possible size of a set $Y$ of points of $X$ that don't see each other through $X$.
\footnote{The term \emph{$m$-convexity} coined by Valentine~\cite{Valentine57} is closely related to the \emph{invisibility number}: a set $X$ is $m$-convex if and only if its invisibility number is $< m$.}
%The notion of \emph{invisibility number} coincides with the notion of \emph{$m$-convexity}: A set is called $m$-convex if its invisibility number is $m-1$.} 
For closed sets, a series of bounds was obtained, culminating with the bound $\gamma(X) \leq 18 \omega(X)^3$ proved by Matou\v{s}ek and Valtr~\cite{MatousekV99}. For general sets, an easy example shows that $\gamma$ is not bounded in terms of $\omega$: If $X$ is obtained from the unit disc by removing the vertices of a regular $n$-gon concentric with the disc and placed close to its boundary, 
then $\omega(X)=3$ while it is easy to show that $\gamma(X) \geq \lceil \frac{n}{2} \rceil +1$~\cite{MatousekV99}.\footnote{We note that in~\cite{CibulkaKKMRV17} it was claimed that in this case, $\gamma(X) = \lceil \frac{n}{2} \rceil +1$; the result was attributed there to~\cite{MatousekV99}. As we wrote in the text, the claim proved in~\cite{MatousekV99} is $\gamma(X) \geq \lceil \frac{n}{2} \rceil +1$. The exact value is $\gamma(X) = \lfloor \frac{n+5}{2} \rfloor - \delta(n)$ where $\delta(n)=1$ for $n=0,1,3$ and $\delta(n)=0$ otherwise, as follows from Theorem~\ref{thm:main}.} 

Motivated by this example, Matou\v{s}ek and Valtr suggested to use the number of isolated points in $\mathbb{R}^2 \setminus X$, which they denoted by $\lambda(X)$, to bound $\gamma(X)$. They showed that $\gamma(X) \leq \omega(X)^4+\lambda(X)\omega(X)^2$, and that this bound is sharp up to a multiplicative factor of $\omega(X)$. In addition, they raised the conjecture that in the plane, $\gamma(X)$ can be bounded in terms of another parameter -- $\chi(X)$, defined as the \emph{chromatic number} of the \emph{invisibility graph} $G_X$ of $X$. This is the graph whose vertex set is $X$ and whose edges connect points of $X$ that do not see each other via $X$. (Clearly, $\omega(X)$ is the \emph{clique number} of $G_X$ and thus $\omega(X) \leq \chi(X) \leq \gamma(X)$ for every $X$).

In~\cite{LawrenceM09}, Lawrence and Morris initiated the study of the convexity number of co-finite sets in $\mathbb{R}^d$ -- namely, sets of the form $X=\mathbb{R}^d \setminus P$, where $P$ is a finite set. Obviously, for such sets we have $\lambda(X)=|P|$, as all points in the complement of $X$ are isolated. The authors of~\cite{LawrenceM09} focused on \emph{lower bounds} on $\gamma(X)$ in terms of $|P|=\lambda(X)$, which they formulated as upper bounds on $|P|$ in terms of $\gamma(X)$. In the same direction, they showed that in the plane, $|P|$ is bounded from above even in terms of $\chi(X)$, and consequently, by the aforementioned result of~\cite{MatousekV99}, for such sets, $\gamma(X)$ is bounded in terms of $\chi(X)$ as was conjectured in~\cite{MatousekV99}. Cibulka et al.~\cite{CibulkaKKMRV17} generalized the bounds of Lawrence and Morris, proving the conjecture of~\cite{MatousekV99} for any $X \subset \mathbb{R}^2$. Very recently, Keller and Perles~\cite{KellerP25} expanded the study initiated by Lawrence and Morris, obtaining a series of structural results on co-finite sets in $\mathbb{R}^d$.    

\medskip

In this paper, we follow up on the study of co-finite sets in $\mathbb{R}^2$ initiated by Lawrence and Morris~\cite{LawrenceM09}, but we  concentrate on the other end of the spectrum -- \emph{upper bounds} on $\gamma(X)$ in terms of $|P|$. We
study several variants of the problem. Regarding the set $P$, besides the `standard' setting in which $P$ is a set of $n$ points in general position in the plane, we study the \emph{convex} setting in which the points of $P$ are assumed to be in convex position.  
This question was explicitly asked by Lawrence and Morris~\cite[Problem~5]{LawrenceM09}. Regarding the convexity number, besides the `standard' setting we study the \emph{disjoint} setting in which the convex sets covering $\mathbb{R}^2 \setminus P$ have to be pairwise disjoint. This variant is motivated by the relation of our problem to Helly-type theorems for unions of convex sets described by Matou\v{s}ek~\cite[Section~2]{Matousek97}, as in such Helly-type theorems, the `disjoint' setting is probably the more natural one (see~\cite{Amenta94,GrunbaumM61}). Finally, regarding the covered area, besides the `standard' setting where the whole complement $\mathbb{R}^2 \setminus P$ must be covered, we study the \emph{encapsulation} setting where it is sufficient to cover a pointed neighborhood of each point in $P$ (see Definition \ref{def:star}). This setting lends itself more easily to inductive proofs, and was studied in~\cite{KellerP25}. 

Combinations of these settings make up eight problems. To define them formally, we use the following notations. For a finite set $P$ of points in a general position in the plane, let: 
\begin{itemize}
	\item $cov(P)$= The smallest number of convex subsets that cover $\Re^2 \setminus P$.
		\item $cov_{\circ}(P)$= The smallest number of pairwise disjoint convex subsets that cover $\Re^2 \setminus P$.
			\item $enc(P)$= The smallest number of convex subsets that encapsulate $P$.
				\item $enc_{\circ}(P)$= The smallest number of pairwise disjoint convex subsets that encapsulate $P$.
\end{itemize}
For $n \in \mathbb{N}$, we define $cov(n)=\max_{P:|P|=n}(cov(P))$. The notations  $cov_{\circ}(n)$, $enc(n)$ and $enc_{\circ}(n)$ are defined similarly. %For the setting where $P$ is a set of points in convex position in the plane, we use the notation 
In the setting where the set $P$ is in convex position we add superscript $c$, e.g.,
\[
cov^c(n)=\max_{P:|P|=n, P \mbox{ is in convex position}}(cov(P)).
\]
%The notations $cov_{\circ}^c(n)$, $enc^c(n)$ and $enc_{\circ}^c(n)$ are defined similarly. 
%We aim at determining the eight quantities 
%$cov(n), cov_{\circ}(n), enc(n), enc_{\circ}(n), cov^c(n), cov_{\circ}^c(n), enc^c(n),$ and $enc_{\circ}^c(n)$ as a function of $n$.

\medskip We obtain the following results.

\begin{theorem}
\label{thm:main}
    In the above notation, the following holds:
    \begin{enumerate}
        \item For any set $P$ of $n$ points in convex position in the plane, 
        \[
        enc(P)=cov(P)=\left\lfloor\frac{n+5}{2} \right\rfloor- \delta(n),
        \]
        where $\delta(n)=1$ for $n=0,1,3$ and $\delta(n)=0$ otherwise, and 
        \[
        enc_{\circ}(P)=cov_{\circ}(P)=\left\lfloor\frac{2n+5}{3} \right\rfloor.
        \]
        Consequently, $enc^c(n)=cov^c(n)=\lfloor\frac{n+5}{2}\rfloor- \delta(n)$, and $enc^c_{\circ}(n)=cov^c_{\circ}(n)=\lfloor\frac{2n+5}{3}\rfloor$.

        \item For any set $P$ of $n$ points in general position in the plane, 
        \[
        enc(P),cov(P) \leq \frac{7n}{11}+4, \qquad \mbox{and} \qquad enc_{\circ}(P),cov_{\circ}(P) \leq \left\lfloor\frac{2n+5}{3} \right\rfloor.
        \]
        By the first part of the theorem, this implies 
        \[
        \left\lfloor\frac{n+5}{2} \right\rfloor- \delta(n) \leq
        enc(n),cov(n) \leq \frac{7n}{11}+4,
        \]
        where $\delta(n)=1$ for $n=0,1,3$ and $\delta(n)=0$ otherwise, and 
        \[
        enc_{\circ}(n)=cov_{\circ}(n)=\left\lfloor\frac{2n+5}{3} \right\rfloor.
        \]
%
%        \item $enc^c(n)=cov^c(n)=\lfloor\frac{n+5}{2}\rfloor- \delta(n)$;
%
%        \item $enc_{\circ}^c(n) =enc_{\circ}(n) =cov_{\circ}^c(n) =cov_{\circ}(n)= \lfloor\frac{2n+5}{3}\rfloor$, 
    \end{enumerate}
%where $\delta(n)=1$ for $n=0,1,3$ and $\delta(n)=0$ otherwise.    
\end{theorem}
In particular, this fully resolves the \emph{disjoint} setting and the \emph{convex} setting of the problem, the latter resolving the problem of Lawrence and Morris~\cite{LawrenceM09}.  

We note that while in our extremal results there is no difference between the `covering' setting and the `encapsulation' setting, these two notions can differ very significantly for specific sets of points.  
Consider, for example, the set 
\[
P=\{  (m,n): m,n \in \mathbb{N}, 1 \leq m,n \leq K, \mbox{ and not both }m \mbox{ and }n \mbox{ are even} \},
\] where $K$ is a large integer. We have $cov(P) \geq (\frac{K}{2})^2$ since no two distinct points of the type $(2 \ell, 2k)$ ($0 \leq k,\ell \leq \lfloor \frac{K+1}{2} \rfloor$) see each other via $\Re^2 \setminus P$. Indeed, some point of $P$ lies in the segment connecting any two such points.
On the other hand, $enc(P) \leq 2K+2$. A corresponding cover of $\Re^2 \setminus P$ consists of the vertical and the horizontal strips 
\begin{align*}
\{(i,i+1) \times \Re \}&_{1 \leq i \leq K-1} \bigcup \{ \Re \times(i,i+1)  \}_{1 \leq i \leq K-1} \\
&\cup \{(-\infty,1) \times \Re \} \cup \{(K,\infty) \times \Re \} \cup \{\Re \times(-\infty,1)\} \cup \{ \Re \times(K,\infty)\}.
\end{align*}

The way in which we obtain the results is demonstrated in Figure~\ref{fig:fig14}. The trivial relations stating that the convex variant of each parameter is no larger the general variant, the encapsulation variant is no larger than the covering variant, and the disjoint variant is no smaller than the corresponding general variant, are depicted by arrows, where the relation $\leq$ leads from the tail of each arrow to its head. 
%This figure also demonstrates our main results. 
Our results are shown in the figure, and the assertion of Theorem~\ref{thm:main} follows from them using the trivial relations encoded by the arrows.

The remaining open problem is to determine $enc(n)$ and $cov(n)$; 
%and in particular, to understand whether there is a difference between the convex setting and the `general position' setting. In fact, we could not find a set $P$ of $n$ points in general position in the plane whose convexity number is $> \lfloor \frac{n+5}{2} \rfloor - \delta(n)$. 
the gap between the lower and upper bounds we obtained for them is significant.

\begin{figure}[ht]
	\begin{center}
		\scalebox{1.1}{
			\includegraphics[width=0.8\textwidth]{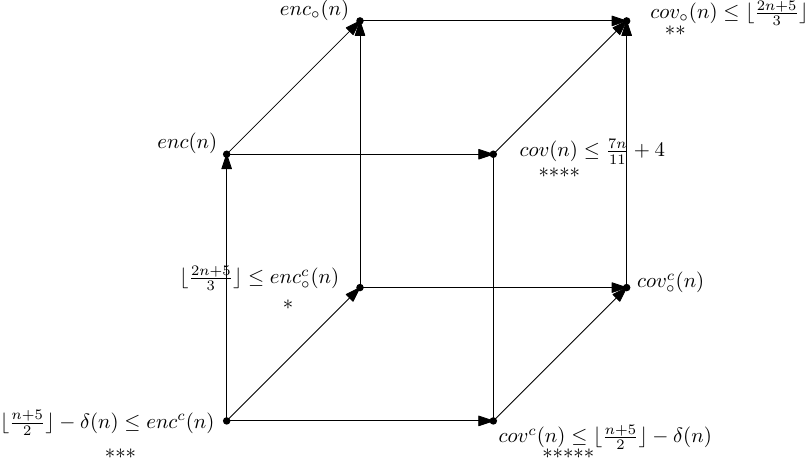}
		}
\caption{A diagram of our results.
%This figure demonstrates 
%the relations between the 8 parameters defined in the introduction.
%the results of the paper.
Each arrow represents a trivial `$\leq$' relation, where the parameter that corresponds to the tail of the arrow is smaller than or equal to the parameter the corresponds to the head of the arrow. 
%Moreover, the figure presents a diagram of our results. 
Thm *=Proposition \ref{cl:tightnesUB}, Thm**=Proposition \ref{cl:g_1UB}, Thm***=Theorem \ref{thm:encaps_comp_of_n}, Thm****=Theorem \ref{thm:cov(n)UB}, Thm*****=Theorem \ref{thm:cover_comp_of_n}.
%        In particular, it follows that $enc_{\circ}^c(n)=cov_{\circ}(n)=cov_{\circ}^c(n)=enc_{\circ}(n)=\lfloor\frac{2n+5}{3}\rfloor$ and $enc^c(n)=cov^c(n)=\lfloor\frac{n+5}{2}\rfloor- \delta(n).$ There remains still the gap $\lfloor\frac{n+5}{2}\rfloor- \delta(n) \leq cov(n) \leq \frac{7n}{11}+4$.}
The assertions of Theorem~\ref{thm:main} follow from these results via the `arrow' relations.
}
		\label{fig:fig14}
	\end{center}
\end{figure}

\medskip

The rest of the paper is organized as follows. The encapsulation lower bounds (Theorem \ref{thm:encaps_comp_of_n} and Proposition \ref{cl:tightnesUB} in Figure \ref{fig:fig14}) are presented in Section \ref{sec:enc}. The covering upper bounds (Theorem \ref{thm:cover_comp_of_n}, Theorem \ref{thm:cov(n)UB} and Proposition \ref{cl:g_1UB} in Figure \ref{fig:fig14}) are presented in Section \ref{sec:cov}.

\section{Preliminaries}\label{sec:preliminaries}

For a set $S \subset \Re^d$, denote by $\mbox{cl}(S)$ and $\mbox{int}(S)$ the topological closure and interior of $S$, respectively. The convex hull of $S$ is denoted by $\mbox{conv}(S)$. 

\begin{definition}\label{def:star}
A point $p \in \Re^2$ is encapsulated by the sets $K_1,\ldots,K_n$ if there exists some neighborhood $N$ of $p$ such that $(\bigcup_{i=1}^n K_i )\cap N = N \setminus\{p\}$.
\end{definition}

\begin{definition}\label{def:touch}
	For $K \subset \Re^d, p \in \Re^d$, we say that $p$ touches $K$ (or $K$ touches $p$), if $p \in \mbox{cl}(K) \setminus K$.
\end{definition}
Given a finite set $P$ of points in $\Re^2$ and a set  $K \subset \Re^d$, let $touch_P(K)=\{p \in P : p \mbox{ touches } K\}$. In cases where the dependence on $P$ is clear from the context we simply write $touch(K)$.

\begin{observation}\label{obs:prem}
Let $P$ be a set of points that are encapsulated by $t$ convex sets. If $|P|=1$ then $t \geq 2$ and if $|P|=2$ then $t \geq 3$.
\end{observation}

For $a,b \in \Re^d$ denote by $\ell(a,b)$ the line through $a$ and $b$.

\iffalse
\begin{theorem}\label{thm:comp_of_n}
	$$cov_{\circ}(n) = \lfloor \frac{2n+5}{3} \rfloor,$$
	and a tightness construction is obtained by $n$ points in convex position in the plane.
\end{theorem}
 Theorem \ref{thm:comp_of_n} follows immediately from the combination of Claims \ref{cl:g_1UB},\ref{cl:tightnesUB} below.
 \fi

%\noindent \textbf{Open Problem:} What is the asymptotic behavior of the function $cov(n)$?  
\section{Lower Bounds}\label{sec:enc}

In this section we prove the two encapsulation lower bounds depicted in Figure \ref{fig:fig14}. First we prove Theorem \ref{thm:encaps_comp_of_n}, and then we prove Proposition \ref{cl:tightnesUB}.
\begin{theorem}\label{thm:encaps_comp_of_n}
	\[
    \left\lfloor \frac{n+5}{2} \right\rfloor -\delta(n) \leq enc^c(n) ,
    \]
	where $\delta(0)=\delta(1)=\delta(3)=1 $ and  $\delta(n)=0$ for all $n \neq 0,1,3$. Namely, at least $\lfloor \frac{n+5}{2} \rfloor -\delta(n)$ convex sets are required to encapsulate $n$ points in convex position in the plane.
	%A corresponding construction is obtained by $n$ points in convex position in the plane. The construction demonstrates even a stronger property: $\lfloor \frac{n+5}{2} \rfloor -\delta(n)$ convex sets are sufficient to \emph{cover} the complement of $n$ points in convex position in the plane.
\end{theorem}

\begin{proof}
 %First, we prove the inequality in the statement of the theorem. 
 The cases $n=1,2,3$ are trivial. For $n=4$ we shall prove that three convex sets are not sufficient to encapsulate four points $\{a_0,a_1,a_2,a_3\}$ in convex position in the plane.\footnote{We note that a slightly different argument for a similar problem with $n=4$ is given in \cite[Theorem 1]{LawrenceM09}.} Assume to the contrary that three convex sets $A,B,C$ suffice. Consider the eight short segments on $\ell(a_i,a_{i+1})$ (where the indices are taken modulo 4) that emanate outside of $\mbox{conv}(a_0,\ldots,a_3)$ (these short segments are drawn in regular lines in Figure \ref{fig:fig7}).
 \begin{figure}[ht]
 	\begin{center}
 		\scalebox{0.6}{
 			\includegraphics[width=0.8\textwidth]{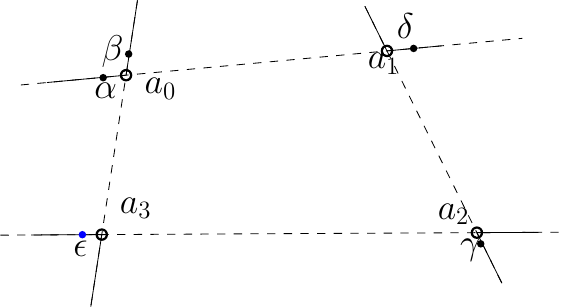}
 		}
 		\caption{An illustration for the proof of Theorem \ref{thm:encaps_comp_of_n}.}
 		\label{fig:fig7}
 	\end{center}
 \end{figure}
By the pigeonhole principle, three of them contain a point of $A$ (w.l.o.g.) very close to the corresponding $a_i$. 
\begin{itemize}
    \item If these three points are $\alpha,\beta$ and $\gamma$ (in the notations of Figure \ref{fig:fig7}), then $a_0 \in \mbox{conv}(\alpha,\beta,\gamma ) \subset A$, a contradiction.
    \item If these 3 points are $\beta,\gamma, \delta$ then $a_1,a_2 \in \mbox{conv}(\beta,\gamma,\delta ) \subset A$, a contradiction.
    \item  If these three points are $\alpha,\beta$ and $\epsilon$ then we need one more step for the contradiction: Consider the line $\ell(a_1,a_2)$. By Observation \ref{obs:prem}, all three sets $A,B,C$ are needed to encapsulate $a_1,a_2$ on this line. Whether $a_1 \in touch (A) $ or $a_2 \in touch(A)$ we have $a_0 \in A$, a contradiction.
    \item All other cases are symmetric or similar.
\end{itemize}

Now we pass to $n \geq 5$. Let $f(n)= \lfloor \frac{n+5}{2} \rfloor$. Let $P$ be a set of $n$ points $\{a_0,\ldots,a_{n-1}\}$ in convex position in the plane. Consider a family $\mathcal{K}$ of convex sets that encapsulate $P$. We shall prove that $|\mathcal{K}| \geq f(n)$. To this end we consider two cases:

\noindent \textbf{Case A:} There exists some $x \in \mbox{int(conv}P)$ that is contained in at least two sets in $\mathcal{K}$.

\noindent \textbf{Case B:} No point in $ \mbox{int(conv}P)$ is contained in two sets in $\mathcal{K}$.

\medskip

We handle each case separately.

\noindent \textbf{Case A:} Consider the short segments on $\ell(x,a_i)$ outside $\mbox{conv}P$ (the regular segments in Figure \ref{fig:fig8}).
 \begin{figure}[ht]
	\begin{center}
		\scalebox{0.6}{
			\includegraphics[width=0.8\textwidth]{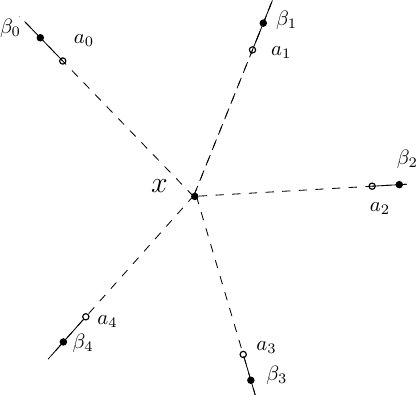}
		}
		\caption{An illustration for the proof of Theorem \ref{thm:encaps_comp_of_n}.}
		\label{fig:fig8}
	\end{center}
\end{figure}
Let $\beta_i$ be a point on such a segment very close to $a_i$. We claim that no three $\beta_i$'s are contained in the same $K \in \mathcal{K}$. Indeed, assume $\beta_i, \beta_j,\beta_k \in K$. If $x,\beta_i, \beta_j$ lie on the same line, then $a_i,a_j \in \mbox{conv}(\beta_i, \beta_j) \subset K$, a contradiction. Otherwise, either all 3 rays from $x$ to $\beta_i, \beta_j,\beta_k$ (in this cyclic order) are contained in a half-plane or not. In the first case (see Figure \ref{fig:fig9}(a)) $a_j \in \mbox{conv}(\beta_i, \beta_j,\beta_k) \in K$, a contradiction. In the second case (see Figure \ref{fig:fig9}(b)) $a_i,a_j,a_k \in \mbox{conv}(\beta_i, \beta_j,\beta_k) \in K$, a contradiction.
 \begin{figure}[ht]
	\begin{center}
		\scalebox{0.8}{
			\includegraphics[width=0.8\textwidth]{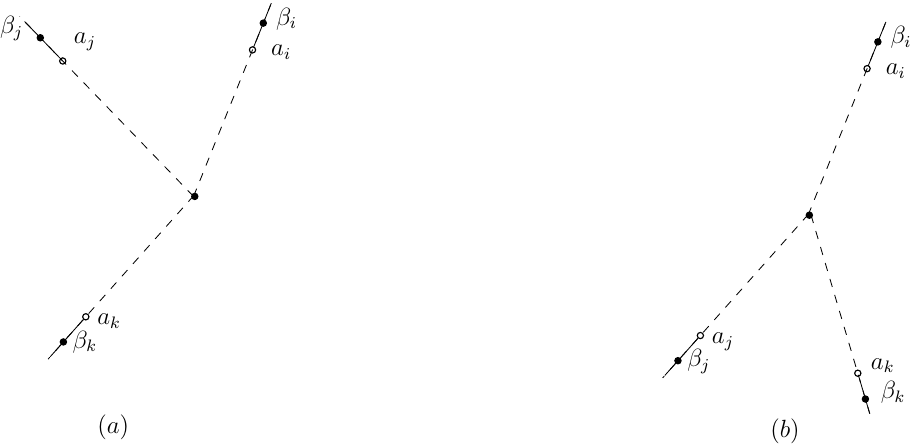}
		}
		\caption{An illustration for the proof of Theorem \ref{thm:encaps_comp_of_n} -- case A.}
		\label{fig:fig9}
	\end{center}
\end{figure}
Therefore, since no convex set can be used 3 times, by the pigeonhole principle we have used so far at least $ \lceil \frac{n}{2} \rceil$ sets. None of these $ \lceil \frac{n}{2} \rceil$ sets contains $x$ (since otherwise some $a_i$ is contained in $K$). Then $|\mathcal{K}| \geq  \lceil \frac{n}{2} \rceil+2 \geq \lfloor \frac{n+5}{2} \rfloor = f(n)$, as stated.

\medskip

\noindent \textbf{Case B:} Since for $n>5$ we have $f(n-2)=f(n)-1$, and for $i \geq 3$ we have $f(2i)=f(2i-1)$, it is sufficient to prove the assertion only for odd $n$, where the base case $n=5$ will be discussed in Claim \ref{cl:n=5} below. For the induction step, let $n>5$ be an odd integer and assume correctness for $n-2$. 

If some $K \in \mathcal{K}$ touches at most 2 points of $P$, then by the induction hypothesis at least $f(n-2)$ convex sets are needed to encapsulate all other $n-2$ points. Hence $|\mathcal{K}| \geq 1+f(n-2)=f(n)$ (in the right equality we use the assumption $n>5$).

From now on we assume that each $K \in \mathcal{K}$ touches at least three points in $P$. Let a \emph{span} of $K \in \mathcal{K}$ be a shortest arc on the boundary of $\mbox{conv}P$ that contains all points in $touch(K)$. The \emph{length} of the span is the number of points in $P$ it contains. Let $K \in \mathcal{K}$ be a set with a shortest span, $\gamma$. Then $K$ touches all the points in $P \cap \gamma$, since otherwise some inner $p \in \gamma \cap P$ touches another $K' \in \mathcal{K}$ whose span is at least as long. But then $K \cap K' \neq \emptyset$ and
we contradict the assumption of Case B. 

In particular, $K$ touches three consecutive vertices of $\mbox{conv}P$, say $a_1,a_2,a_3$. Since by Observation \ref{obs:prem} no point is encapsulated by a single convex set, $a_2$ touches some other $K'' \in \mathcal{K}$. But since $K''$ touches at least three points in $P$, we have $\mbox{int}(K) \cap  \mbox{int}(K') \cap \mbox{int}(\mbox{conv}P) \neq \emptyset$ in contradiction to the assumption of Case B.

\medskip

To complete the proof of Case B, we have to prove the induction basis for $n=5$. For the sake of convenience this is done in Claim \ref{cl:n=5} below. Up to this induction basis we completed the proof of Theorem \ref{thm:encaps_comp_of_n}.
\end{proof}

\medskip

It remains to prove the induction basis for $n=5$. 
\begin{claim}\label{cl:n=5}
	Let $P=\{a_1,\ldots,a_5\}$ be a set of 5 points in this cyclic order in convex position in the plane, and let $K_1,\ldots,K_t$ be $t$ convex sets that encapsulate $P$, where no point in $\mbox{int} (\mbox{conv}P)$ is contained in two $K_i$'s. Then $t \geq 5$.
\end{claim}
\begin{proof}[Proof of Claim \ref{cl:n=5}]
	Assume to the contrary that $t < 5$. We use several observations:
	\begin{observation}\label{obs:1}
		Each $K_i$ touches at least two $a_j$'s. 
	\end{observation}
Indeed, otherwise there are 4 points that are encapsulated by $t-1$ convex sets. As we proved at the beginning of the proof of Theorem \ref{thm:encaps_comp_of_n}, it follows that $t-1 \geq 4$, a contradisction.
\begin{observation}\label{obs:2}
	Some $K_i$ touches at least 3 of the $a_j$'s.
\end{observation} 
\begin{proof}
By Observation \ref{obs:1} the number of touchings of $K_i$'s and $a_j$'s is at least $2 \times 5=10$. If each $K_i$ touches exactly 2 points then $t=5$, a contradiction. Otherwise, by double counting, the assertion of Observation \ref{obs:2} follows.
\end{proof}
\begin{observation}\label{obs:3}
	For every $1 \leq i_1<i_2 \leq t$, $|touch(K_{i_1}) \cup touch(K_{i_2})|>3$.
\end{observation}
\begin{proof}
	Otherwise, $|P \setminus (touch(K_{i_1}) \cup touch(K_{i_2}))| \geq 2$, hence by Observation \ref{obs:prem}, at least 3 convex sets are needed to encapsulate $P \setminus (touch(K_{i_1}) \cup touch(K_{i_2}))$, and together with $K_{i_1},K_{i_2}$ we have $t \geq 5$, a contradiction.
\end{proof}
\begin{observation}\label{obs:4}
	Each $a_j$ touches some $K_i$ with $|touch(K_i) | \geq 3$.
\end{observation}

	Indeed, by Observation \ref{obs:prem}, $a_j$ touches at least two $K_i$'s, and if both touch at most 2 points, it contradicts Observation \ref{obs:3}. 

\medskip

Now we are ready to continue with the proof of Claim \ref{cl:n=5}. Assume that $|touch(K_1)|=\max\{|touch(K_1)|,\ldots,|touch(K_t)|\}$. By Observation \ref{obs:2}, $|touch(K_1)| \geq 3$, hence we consider 3 cases:

\medskip

\noindent \textbf{Case 1: $|touch(K_1)|=5$.} Then since no two $K_i$'s intersect in $\mbox{int}(\mbox{conv}P)$, by Observation \ref{obs:1} each of $K_2,\ldots,K_t$ touches two consecutive $a_j$'s. Since by Observation~\ref{obs:prem} each $a_j$ touches at least two $K_i$'s, the family $K_2,\ldots,K_t$ contains at least 3 sets, but then two $K_i$'s touch together 3 points, contradicting Observation \ref{obs:3}.

\medskip

\noindent \textbf{Case 2: $|touch(K_1)|=4$.} Assume that $touch(K_1)=\{a_2,a_3,a_4,a_5\}$. A set $K_i$ that touches $a_1$ cannot touch $a_3$ or $a_4$ since no two convex sets intersect in $\mbox{int}(\mbox{conv}P)$. By Observation \ref{obs:prem} at least 3 convex sets are needed to encapsulate $a_3,a_4$. At least 2 other covex sets are needed to encapsulate $a_1$, thus in total $t \geq 5$, a contradiction.

\medskip

\noindent \textbf{Case 3: $|touch(K_1)|=3$.} If $touch(K_1)$ contains 3 consecutive points, say $a_1,a_2,a_3$ then by Observation \ref{obs:prem} $a_2$ touches another set, say $K_2$. Since $K_1$ and $K_2$ do not intersect in $\mbox{int}(\mbox{conv}P)$, either $touch(K_2)=\{a_1,a_2\}$ or $touch(K_2)=\{a_2,a_3\}$. Then $|touch(K_{1}) \cup touch(K_{2})|=3$ contradicting Observation \ref{obs:3}.
The remaining setting of case 3 is where $touch(K_1)$ contains 3 non-consecutive points, say $a_1,a_3,a_4$, and we can also assume that each other $K_i$ with $|touch(K_i)|=3$ touches 3 non-consecutive points (otherwise just replace the corresponding set with $K_1$). Then by Observation \ref{obs:4}, $a_2$ touches some $K_i$ with $|touch(K_i)|=3$, but then $K_1$ and $K_i$ intersect in $\mbox{int}(\mbox{conv}P)$, a contradiction.
This completes the proof of Claim \ref{cl:n=5}.
\end{proof}

\medskip

Now we turn to the second result of this section.

\begin{proposition}\label{cl:tightnesUB}
	$enc_{\circ}^c(n) \geq \lfloor \frac{2n+5}{3} \rfloor$. In other words, if a set $P=\{p_1,\ldots,p_n\} \subset \Re^2$ of $n$ points in convex position is encapsulated by the pairwise disjoint convex sets $\mathcal{K}=\{K_1,\ldots,K_t\}$, then $t \geq \lfloor \frac{2n+5}{3} \rfloor$.
\end{proposition}

%Since $enc_{\circ}(n) \leq cov_{\circ}(n)$, Proposition \ref{cl:tightnesUB} implies $cov_{\circ}(n) \geq \lfloor \frac{2n+5}{3} \rfloor$, hence together with Proposition \ref{cl:g_1UB}, Theorem \ref{thm:comp_of_n} follows.

\begin{proof}[Proof of Proposition \ref{cl:tightnesUB}]
	Let $f(n)=\lfloor \frac{2n+5}{3} \rfloor$. Then $\forall n \geq 3, f(n)=2+f(n-3)$, and the sequence $\{f_n\}_{n=0}^{\infty}$ is $\langle 1,2,3,3,4,5,5,6,7,7,\ldots  \rangle$. The proof is by induction on $n$, where the cases $n=0,1,2,3$ are trivial.
	
	We start with several reductions: first, we can assume that each $K_i$ touches some point in $P$ (otherwise we can simply discard $K_i$), and each $p_j \in P$ touches at least two $K_i$'s (by Observation \ref{obs:prem}). Moreover, if some set $K_i$ touches only one point $p_j \in P$, then $P \setminus \{p_j\}$ is encapsulated by $\mathcal{K} \setminus \{K_i\}$, and by the induction hypothesis $f(n-1) \leq t-1$. Hence $f(n) \leq f(n-1)+1 \leq t$ and we are done. On the other hand, if every set in $\mathcal{K}$ touches just two points, then since every point in $P$ touches at least two $K_i$'s, we have $n \leq t$. Since $\forall n \geq 3 , f(n) \leq n $, we are done again. Hence, from now on we assume that each $K_i$ touches at least two points of $P$, and at least one $K_i$ touches 3 or more points.
	
	For each $K_i \in \mathcal{K}$ let $touch(K_i)=\{p \in P: p \mbox{ touches }K_i\}$. In the arguments below we use the following observation:
	
	\begin{observation}\label{obs:touch<=3}
		In the notations of Proposition \ref{cl:tightnesUB}, if for some $1 \leq i<j \leq n$, $|touch(K_i) \cup touch(K_j)| \leq 3$, then $f(n) \leq t$ and we are done.
	\end{observation}    

\begin{proof}[Proof of Observation \ref{obs:touch<=3}]
	If (w.l.o.g.) $touch(K_i) \cup touch(K_j) \subseteq \{p_1,p_2,p_3\}$ then the $t-2$ sets in $\mathcal{K} \setminus\{K_i,K_j\}$ encapsulate the $n-3$ (or more) points in $P \setminus \{p_1,p_2,p_3\}$, and by the induction hypothesis $f(n-3) \leq t-2$. Therefore $f(n)=2+f(n-3)\leq 2 + (t-2)=t$ and we are done again.
\end{proof} 
 
We say that $K_i \in \mathcal{K}$ is \emph{big} if $|touch(K_i)| \geq 3$. Under the reductions above, if $K_i$ is not big then $|touch(K_i)|=2$ and we say that $K_i$ is \emph{small}.
We can assume that each $p_i \in P$ touches some big set $K_j$. Indeed, otherwise $p_i$ touches at least two small sets $K_{j_1},K_{j_2}$. It follows that $|touch(K_{j_1}) \cup touch(K_{j_2})| \leq 3$, and by Observation \ref{obs:touch<=3} we are done.

Like in the proof of Theorem \ref{thm:encaps_comp_of_n}, define the \emph{span} of a big set $K_i$ to be a shortest arc on the boundary of $\mbox{conv}(P)$ that includes $touch(K_i)$. The \emph{length} of the span is the number of points in $P$ it contains. Assume w.l.o.g. that $K_1$ is a big set with the shortest span $\gamma$. Then $K_1$ touches all points in $P \cap \gamma$, since otherwise some inner $p \in \gamma \cap P$ touches another big $K_j$ whose span is at least as long. But then $K_i \cap K_j \neq \emptyset$, a contradiction.

Since $K_1$ is big, $|P \cap \gamma| \geq 3$. We now consider the cases  $|P \cap \gamma| = 3$,  $|P \cap \gamma| =4$ and  $|P \cap \gamma| \geq 5$, and show that in each case the assertion follows by the induction hypothesis.
For technical reasons, we first consider the simple case $P \cap \gamma =P$. In this case $K_1$ touches $p_1,\ldots, p_n$ and since for every $2 \leq i \leq n$ we have $touch(K_i) \geq 2$ and $K_1 \cap K_i =\emptyset$, it follows that each $K_i$ touches two consecutive vertices of $\mbox{conv}P$. W.l.o.g. $K_2$ touches $p_1,p_2$, but since no two points are encapsulated by fewer than three convex sets, some other set in $\mathcal{K}$, say $K_3$, touches $p_1$ or $p_2$. Then $|touch(K_2) \cup touch(K_3)| \leq 3$, and by Observation \ref{obs:touch<=3} we are done.

Hence, from now on we can assume that $|P \cap \gamma| <|P|$.

\medskip

\noindent \textbf{Case 1: $|P \cap \gamma| = 3$.} 

Assume $P \cap \gamma = \{p_1,p_2,p_3\}$ (see Figure \ref{fig:fig3}), namely, $K_1$ touches exactly the points $p_1,p_2,p_3$ of $P$. The point $p_2 $ touches another convex set, say $K_2$. Since $K_1 \cap K_2 = \emptyset$, the set $touch(K_2)$ is either $\{p_1,p_2\}$ or $\{p_2,p_3\}$. But then we are done by applying Observation \ref{obs:touch<=3} with $K_1,K_2,p_1,p_2,p_3$.
\begin{figure}[ht]
	\begin{center}
		\scalebox{0.5}{
			\includegraphics[width=0.8\textwidth]{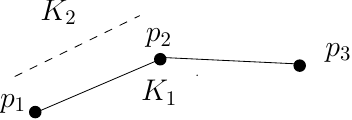}
		}
		\caption{An illustration for Case 1 in the proof of Proposition \ref{cl:tightnesUB} where $touch(K_2)=\{p_1,p_2\}$.}
		\label{fig:fig3}
	\end{center}
\end{figure}

\medskip

\noindent \textbf{Case 2: $|P \cap \gamma| = 4$.} 

Assume $P \cap \gamma = \{p_1,p_2,p_3,p_4\}$ (see Figure \ref{fig:fig4}). Since no point touches just one set in $\mathcal{K}$, the point $p_2$ touches another convex set, say $K_2$. Again, since $K_1 \cap K_2= \emptyset$, the set $touch(K_2)$ is either $\{p_1,p_2 \} $ or $\{p_2,p_3\}$. If $touch(K_2)=\{p_2,p_3\}$ then there exists another set in $\mathcal{K}$, say $K_3$, that touches $p_2$ or $p_3$ (since  two points cannot be encapsulated by fewer than three convex sets), w.l.o.g. $K_3$ touches $p_2$. Since $|touch(K_3)|>1$ and $K_1 \cap K_3= \emptyset$, we have $touch(K_3)=\{p_1,p_2\}$ or $touch(K_3)=\{p_2,p_3\}$ again. But then we are done by Observation \ref{obs:touch<=3} with $K_2,K_3,p_1,p_2,p_3$.

Now we are left with the other possibility, where $touch(K_2)=\{p_1,p_2\}$. The point $p_3$ touches (w.l.o.g.) $K_3$. If $touch(K_3)=\{p_2,p_3\}$ then we are done by Observation \ref{obs:touch<=3}. Therefore we can assume that $touch(K_3)=\{p_3,p_4\}$. Then
$touch(K_1) \cup touch(K_2) \cup touch(K_3) = \{p_1,p_2,p_3,p_4\}$ and by the induction hypothesis, $f(n-4) \leq t-3$ (since the $t-3$ sets in $\mathcal{K} \setminus \{K_1,K_2,K_3\}$ encapsulate the $n-4$ points in $P \setminus \{p_1,p_2,p_3,p_4\}$). Therefore 
\[
f(n) \leq f(n-1)+1 =3 +(f(n-1)-2)=3+f(n-4) \leq 3+(t-3)=t,
\]
as asserted. (Here we used the assumption $n \geq 5$ that holds since $|P \cap \gamma |<|P|$ as discussed above.)
\begin{figure}[ht]
	\begin{center}
		\scalebox{0.6}{
			\includegraphics[width=0.8\textwidth]{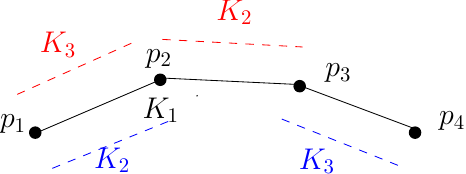}
		}
		\caption{An illustration for case 2 in the proof of Proposition \ref{cl:tightnesUB}. The case $touch(K_2)=\{p_2,p_3\}$ is colored with red, and the case $touch(K_2)=\{p_1,p_2\}$ is colored with green.}
		\label{fig:fig4}
	\end{center}
\end{figure}

\medskip

\noindent \textbf{Case 3: $|P \cap \gamma| \geq 5$.}
Assume $\{p_1,p_2,p_3,p_4,p_5\}\subseteq P \cap \gamma $ (see Figure \ref{fig:fig5}). Since $p_3$ touches more than one set in $\mathcal{K}$, w.l.o.g. $p_3 \in touch(K_2)$. As before, then $touch(K_2)$ is $\{p_2,p_3\}$ or $\{p_3,p_4\}$. W.l.o.g. $touch(K_2)=\{p_3,p_4\}$. Since by Observation \ref{obs:prem} no two points are encapsulated by fewer than three convex sets, we can assume that some other set in $\mathcal{K}$, say $K_3$, touches $p_3$ or $p_4$ or both. Since $K_1 \cap K_3 = \emptyset$ it follows that $|touch(K_3)|=2$ and we have again two convex sets $K_2,K_3$ with $|touch(K_2) \cup touch(K_3)| \leq 3$, and by Observation \ref{obs:touch<=3} we are done. 
\begin{figure}[ht]
	\begin{center}
		\scalebox{0.6}{
			\includegraphics[width=0.8\textwidth]{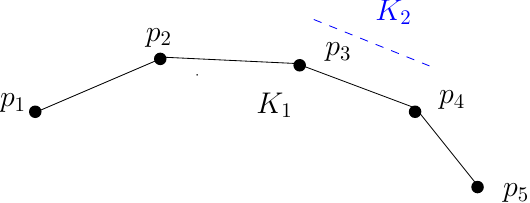}
		}
		\caption{An illustration for case 3 in the proof of Proposition \ref{cl:tightnesUB}.}
		\label{fig:fig5}
	\end{center}
\end{figure}

This completes the proof of Proposition \ref{cl:tightnesUB}.
\end{proof}

\section{Upper Bounds}\label{sec:cov}

In this section we prove the three upper bounds for covering problems depicted in Figure \ref{fig:fig14} -- namely, Theorems~\ref{thm:cover_comp_of_n} and~\ref{thm:cov(n)UB} and Proposition~\ref{cl:g_1UB}. 
%We start with the upper bound in Theorem \ref{thm:cover_comp_of_n}. This upper bound is not only an end in itself but also provides some of the necessary machinery for establishing the subsequent result of Theorem \ref{thm:cov(n)UB} below. 
We begin with the proof of Theorem~\ref{thm:cover_comp_of_n} whose proof also provides some of the necessary machinery for establishing Theorem \ref{thm:cov(n)UB}.

\begin{theorem}\label{thm:cover_comp_of_n}

$cov^c(n) \leq \lfloor \frac{n+5}{2} \rfloor -\delta(n)$ where $\delta(n)$ is as defined in Theorem \ref{thm:encaps_comp_of_n}. Namely,
     $\lfloor \frac{n+5}{2} \rfloor -\delta(n)$ convex sets are sufficient to \emph{cover} the complement of $n$ points in convex position in the plane.
\end{theorem}

\begin{proof}
%To complete the proof of Theorem \ref{thm:encaps_comp_of_n} 
%We show that $f(n)=\lfloor \frac{n+5}{2} \rfloor -\delta(n)$ convex sets suffices to cover the complement of $n$ points in convex position in the plane.
Denote $f(n)=\lfloor \frac{n+5}{2} \rfloor -\delta(n)$. Consider a set $P$ of $n \geq 4 $ points in convex position in $\Re^2$, ordered cyclically $b_0,\ldots,b_{\lfloor \frac{n}{2} \rfloor-1},a_{\lceil \frac{n}{2} \rceil-1},\ldots,a_0$ as in Figure \ref{fig:fig10}(b,c). For each $0 \leq i \leq \lfloor \frac{n}{2} \rfloor-1$, let $H_i^-$ be the half-open half-plane below $\ell(a_i,b_i)$ including the open ray on $\ell(a_i,b_i)$ emanating from $b_i$ to the right, and let $H_i^+$ be the half-open half-plane above $\ell(a_i,b_i)$ including the open ray on $\ell(a_i,b_i)$ emanating from $a_i$ to the left, as in Figure \ref{fig:fig10}(a).
 \begin{figure}[ht]
	\begin{center}
		\scalebox{0.8}{
			\includegraphics[width=0.8\textwidth]{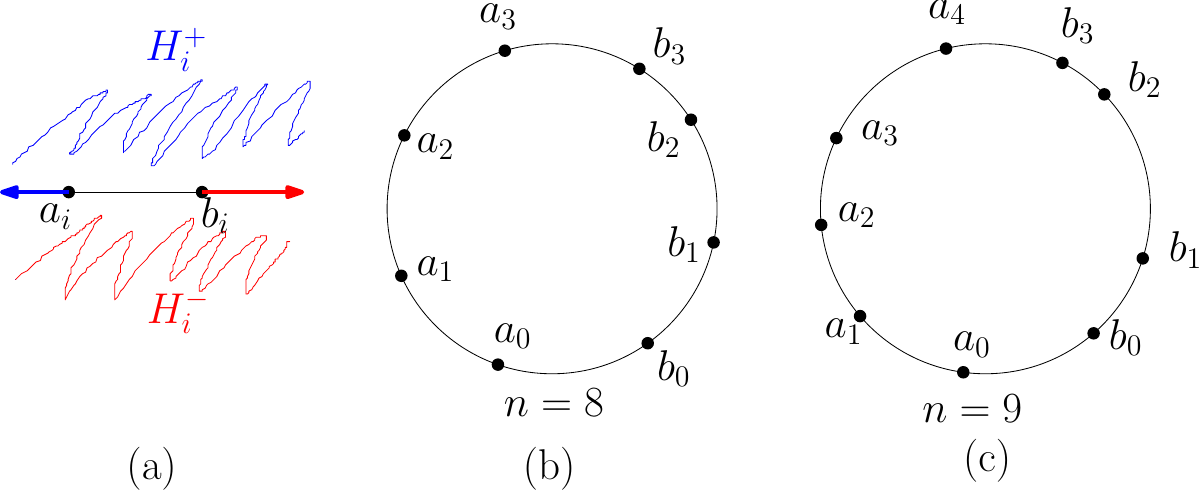}
		}
		\caption{An illustration for the proof of Theorem \ref{thm:cover_comp_of_n}.}
		\label{fig:fig10}
	\end{center}
\end{figure}

We construct $\lfloor \frac{n+5}{2} \rfloor$ convex sets: 
\begin{align*}
	&H_0^- \\
	&H_0^+ \cap H_1^- \\
	&H_0^+ \cap H_1^+ \cap H_2^-\\
	&\ldots \\
	&H_0^+ \cap \ldots \cap H_{\lfloor \frac{n}{2} \rfloor-2}^+ \cap H_{\lfloor \frac{n}{2} \rfloor-1}^-.
\end{align*}
If $n$ is even then the last two convex sets are $H_{\lfloor \frac{n}{2} \rfloor-1}^+$ and $\mbox{int} (\mbox{conv}(\{a_i\}\cup \{b_i\}))$. If $n$ is odd, then the three last convex sets are 
\begin{align*}
	&H_0^+ \cap \ldots \cap H_{\lfloor \frac{n}{2} \rfloor-1}^+ \cap H_{\lfloor \frac{n}{2} \rfloor}^- \\
	&  H_{\lfloor \frac{n}{2} \rfloor}^+ \\
	&\mbox{int} (\mbox{conv}(\{a_i\}\cup \{b_i\})),
\end{align*}
where $ H_{\lfloor \frac{n}{2} \rfloor}^-$ is the half-open half-plane below $a_{\lfloor \frac{n}{2} \rfloor}$ including the left open ray, and $ H_{\lfloor \frac{n}{2} \rfloor}^+$ is the half-open half-plane above $a_{\lfloor \frac{n}{2} \rfloor}$ including the right open ray. It is clear that in both cases, the sets we constructed cover $\mathbb{R}^2 \setminus P$.
\end{proof}

\medskip

\begin{theorem}\label{thm:cov(n)UB}
The complement of any set of $n$ points in general position in the plane can be covered by $ \frac{7n}{11}+4$ convex sets, namely, 
    \[
    cov(n) \leq \frac{7n}{11}+4.
    \]
\end{theorem}

\begin{proof}
    Let $P$ be a set of $n$ points in general position in the plane. If $P$ is in convex position, then by Theorem \ref{thm:cover_comp_of_n}, $cov(P) \leq \lfloor \frac{n+5}{2} \rfloor -\delta(n) \leq \frac{7n}{11}+4$. If $P$ contains $n-1$ points in convex position and a single point $p$ in the interior of their convex hull, then by the proof of Theorem \ref{thm:cover_comp_of_n}, $\Re^2 \setminus (P \setminus \{p\})$ can be covered by $\lfloor \frac{n+5}{2} \rfloor -\delta(n)$ convex sets, such that each point in $\mbox{int}(\mbox{conv}(P\setminus \{p\}))$ is covered twice. Then, by splitting each of the two convex sets that contain $p$ into 2 convex sets, we obtain a cover of $\Re^2 \setminus P$ with $\lfloor \frac{n+5}{2} \rfloor -\delta(n) +2 \leq \frac{7n}{11}+4$ convex sets. From now on we assume that $|P \cap (\mbox{int}(\mbox{conv}P))| \geq 2$. 

    We proceed by induction, where the induction basis is the two `degenerate' settings above. Consider two cases. The simpler one is where there exist three consecutive vertices $a,b,c$ of the boundary of $\mbox{conv}P$, such that some point of $P$ lies inside the triangle $\triangle abc$. Assume w.l.o.g. that $b$ is the highest point of $\mbox{conv}P$, and that the line $\ell(a,c)$ is horizontal (see Figure \ref{fig:fig11}). 
    \begin{figure}[ht]
	\begin{center}
		\scalebox{0.9}{
			\includegraphics[width=0.8\textwidth]{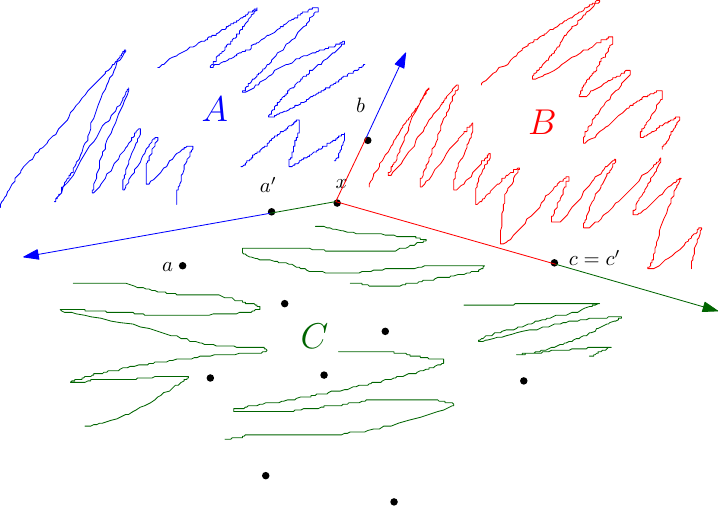}
		}
		\caption{An illustration for the proof of Theorem \ref{thm:cov(n)UB}.}
		\label{fig:fig11}
	\end{center}
\end{figure}

Let $x$ be the highest point in $P \cap \mbox{int}(\triangle abc)$ (if there is more than one highest point, $x$ will be the leftmost one). Let $a',c' \in P$ be two points such that $a',x,c'$ are consecutive vertices of the boundary of $\mbox{conv}(P \setminus \{b\})$. Note that $a'$ can be either $a$ or some higher point in $int (\mbox{conv} (P \setminus \{b\}))$, and similarly for $c'$. In Figure \ref{fig:fig12}, $a' \in \mbox{int}(\mbox{conv} (P \setminus \{b\}))$ and $c'=c$. Note also that $P \cap (\mbox{conv}P \setminus \mbox{conv} (P \setminus \{b\}))=\{b\}$. The rays $\vec{xb}, \vec{xc'},\vec{xa'}$ partition the plane into three convex sets $A,B,C$, leaving $x,a',b$ and $c'$ uncovered, as illustrated in Figure \ref{fig:fig12}. 
   \begin{figure}[ht]
	\begin{center}
		\scalebox{0.9}{
			\includegraphics[width=0.8\textwidth]{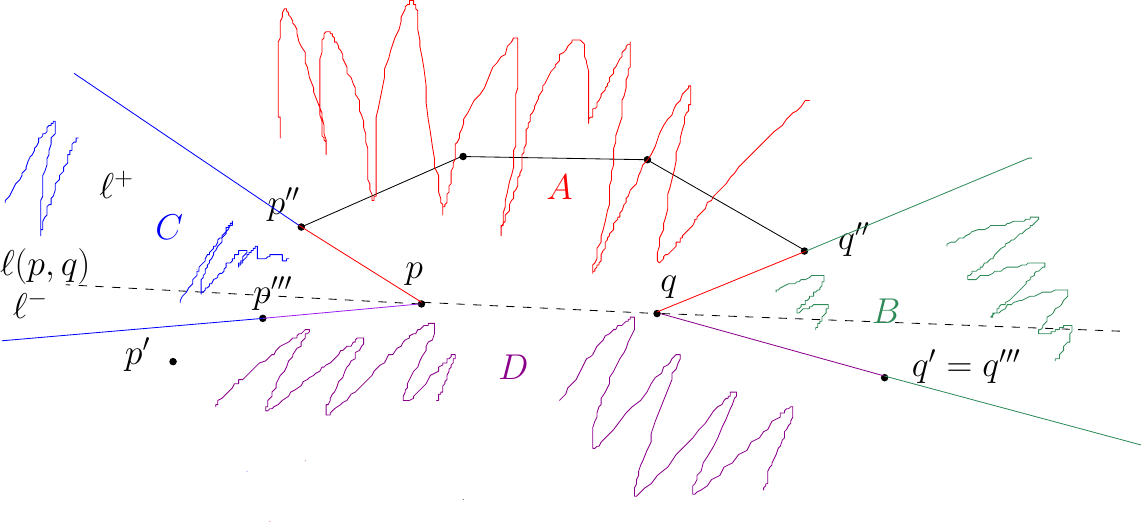}
		}
		\caption{An illustration for the proof of Theorem \ref{thm:cov(n)UB}. The rays $\vec{xb}, \vec{xc'},\vec{xa'}$ except for the points $a',b',c',x$ belong to the corresponding set $A,B,C$ or $D$ as can be seen by the colors.}
		\label{fig:fig12}
	\end{center}
\end{figure}

Since $|C \cap P|=n-4$, by the induction hypothesis $\Re^2 \setminus(P \setminus \{x,a',b',c'\})$ can be covered by $\leq \frac{7(n-4)}{11}+4$ convex sets. Intersecting each of these convex sets with $C$, and adding $A$ and $B$, we obtain a cover of $\Re^2 \setminus P$ by $\leq \frac{7(n-4)}{11}+6 \leq \frac{7n}{11}+4 $ convex sets and we are done.

The remaining case is where each three consecutive vertices of $\mbox{conv}P$ form a triangle whose interior contains no point of $P$. Let $P'= \mbox{int}(\mbox{conv}P)$. By the assumption above $|P'|\geq 2$. Let $[p,q]$ ($p,q \in P'$) be a boundary edge of $\mbox{conv}(P')$. The line $\ell(p,q)$ intersects two non consecutive\footnote{These two edges are indeed non consecutive, since otherwise we have again the first case of a non empty triangle.} boundary edges $[p',p''],[q',q'']$ of $\mbox{conv}P$. (W.l.o.g., all vertices of $P'$ lie below $\ell(p,q)$ as in Figure \ref{fig:fig12}).

Let $\ell^+$ be the closed half-plane above $\ell(p,q)$, and $\ell^-$ be the closed half-plane below $\ell(p,q)$. Let $q''',p''' \in P$ be points such that $p''',p,q,q'''$ are four consecutive vertices of $\mbox{conv} (P \cap \ell^-)$. (The point $p'''$ can be either $p'$ or some inner point of $\mbox{conv}P$ -- the latter is demonstrated in Figure \ref{fig:fig12}, and $q'''$ can be either $q'$ or some inner point of $\mbox{conv}P$ -- the former is demonstrated in Figure \ref{fig:fig12}).

Note that $\angle p''pq+\angle pqq''\geq 180^{\circ}$ or $\angle p'''pq+\angle pqq'''\geq 180^{\circ}$. Let us partition the set $\Re^2 \setminus\{p,q,p'',q'',p''',q'''\}$ into four convex sets $A,B,C,D$ as follows. If $\angle p''pq+\angle pqq''>180^{\circ}$ and $\angle p'''pq+\angle pqq'''>180^{\circ}$ (as in Figure \ref{fig:fig12}) then $A$ is bounded by $\vec{pp''},[p,q]$ and $\vec{qq''}$, $C$ is bounded by $\vec{pp''},\vec{pp'''}$, $B$ is bounded by $\vec{qq''},\vec{qq'''}$, and $D$ is bounded by $\vec{pp'''},[p,q]$ and $\vec{qq'''}$.

If $\angle p'''pq+\angle pqq'''<180^{\circ}$ then the sets $A,B,C,D$ are defined similarly, but $D$ is bounded, as illustrated in Figure \ref{fig:fig13}. Symmetrically, if $\angle p''pq+\angle pqq''<180^{\circ}$ then $A$ is bounded. Anyway, $\mbox{cl}(A) \cap P$ contains only consecutive vertices of the boundary of $\mbox{conv}P$ from $p''$ to $q''$. Moreover, $\mbox{cl}(C) \cap P = \{p,p'',p'''\},\mbox{cl}(B) \cap P = \{q,q'',q'''\}$, and only $D$ contains points of $\mbox{int}(\mbox{conv}P)$. Hence $|D \cap P| \leq n-6$.
   \begin{figure}[ht]
	\begin{center}
		\scalebox{0.9}{
			\includegraphics[width=0.8\textwidth]{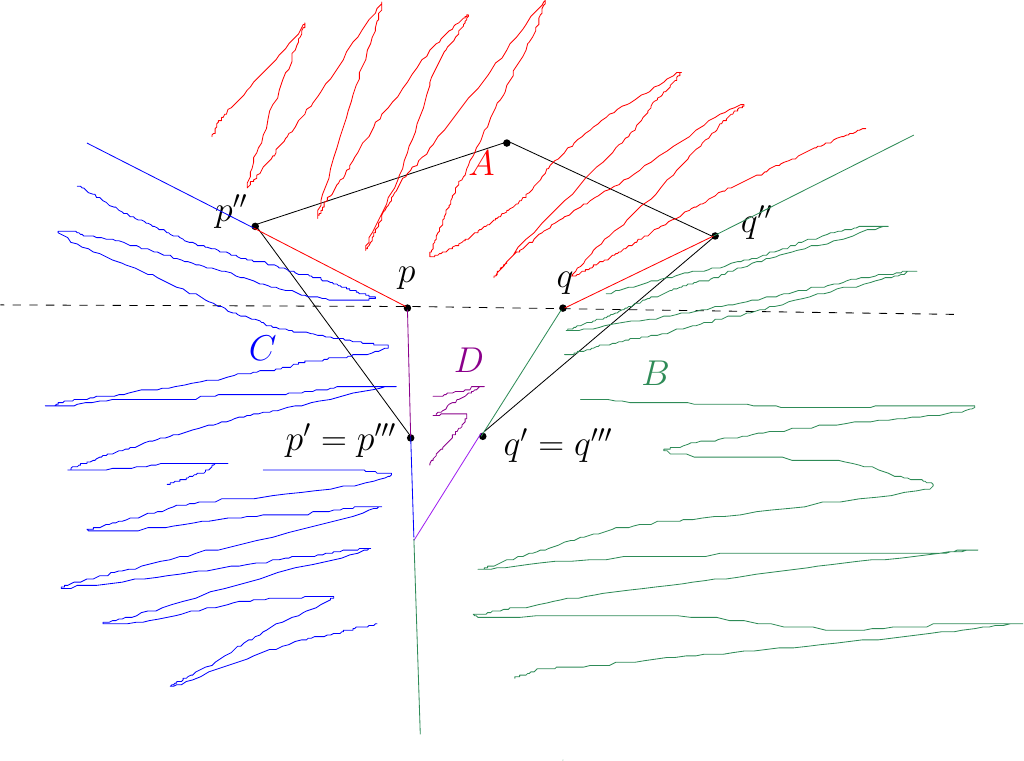}
		}
		\caption{An illustration for the proof of Theorem \ref{thm:cov(n)UB} where $\angle p'''pq+\angle pqq'''<180^{\circ}$.}
		\label{fig:fig13}
	\end{center}
\end{figure}

The remaining part of the proof makes use of the induction hypothesis on $P \cap D$. We intersect each of the convex sets obtained from the induction hypothesis with $D$. This procedure guarantees obtaining a family $\D$ of convex sets.

If $A \cap P = \emptyset$ then $\D \cup \{A,B,C\}$ is a cover of $\Re^2 \setminus P$ by at most $\frac{7(n-6)}{11}+4+3$ convex sets, and since $\frac{7(n-6)}{11}+7< \frac{7n}{11}+4$ we are done.  
If $|A \cap P| =k$ then by Theorem \ref{thm:cover_comp_of_n} (that can be applied here since $A \cap P$ is in convex position), $\Re^2 \setminus (A \cap P)$ can be covered by $\lfloor \frac{k+5}{2} \rfloor - \delta(k)$ convex sets. Let $\A$ be the family of the intersections of these convex sets with $A$. Then $\A \cup \D \cup \{B,C\}$ is a family of at most 
\[
\left(\left\lfloor \frac{k+5}{2} \right\rfloor - \delta(k)\right)+\left(\frac{7(n-6-k)}{11}+4 \right) +2\leq \frac{7n}{11}+4
\]
convex sets that cover $\Re^2 \setminus P$ as needed. (The ratio 7:11 is obtained when $k=5$, for any other value of $k$ the right inequality is strong.)
This completes the proof of Theorem \ref{thm:cov(n)UB}.
\end{proof}

\medskip

\begin{proposition}\label{cl:g_1UB}
\[
     cov_{\circ}(n) \leq \left\lfloor \frac{2n+5}{3} \right\rfloor.
\]
     Namely, the complement of $n$ points in general position in the plane can be covered by $\lfloor \frac{2n+5}{3} \rfloor$ pairwise disjoint convex sets.
\end{proposition}

\begin{proof}[Proof of Proposition \ref{cl:g_1UB}]
We prove the claim by induction on $n$. The inequality is trivial for $n=1$. For $n=2$, a cover of the complement of 2 points in the plane is illustrated in Figure \ref{fig:fig1}$(a)$. The case $n=3$ is illustrated in Figure \ref{fig:fig1}$(b)$.

\begin{figure}[ht]
	\begin{center}
		\scalebox{0.9}{
			\includegraphics[width=0.8\textwidth]{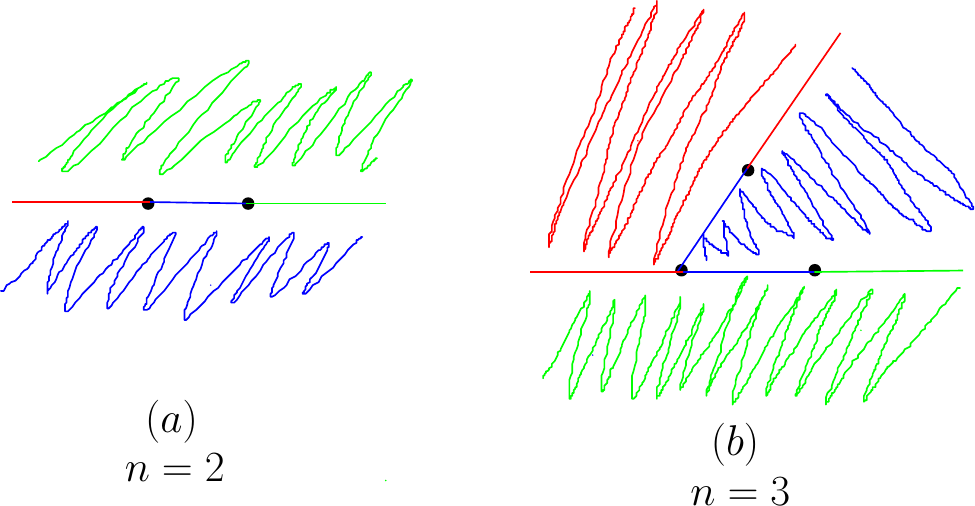}
		}
		\caption{A cover of the complement of $n=2$ and $n=3$ points in $\Re^2 $ with convex sets. Each convex set is colored by a different color.}
		\label{fig:fig1}
	\end{center}
\end{figure}

In the induction step, we shall prove that for $n \geq 3$, $cov_{\circ}(n) \leq 2 +cov_{\circ}(n-3)$, and the assertion will follow. Indeed, given a set $P$ of $n$ points in general position in the plane, let $x,y,z \in P$ be 3 consecutive vertices of $\mbox{conv}P$. Let $A,B$ be convex sets in the complement of $\mbox{int(conv}P)$ as illustrated in Figure \ref{fig:fig2}, and let $C=\Re^2 \setminus(A \cup B \cup \{x,y,z\})$. By the induction hypothesis $\Re^2 \setminus(P \setminus \{x,y,z\})$ can be covered by $cov_{\circ}(n-3)$ convex sets $K_1,\ldots,K_{cov_{\circ}(n-3)}$. Then $A,B,K_1 \cap C,\ldots,K_{cov_{\circ}(n-3)} \cap C$ are $2+cov_{\circ}(n-3)$ convex sets whose union equals $\Re^2 \setminus P$, as asserted.
\end{proof}
\begin{figure}[ht]
	\begin{center}
		\scalebox{0.6}{
			\includegraphics[width=0.8\textwidth]{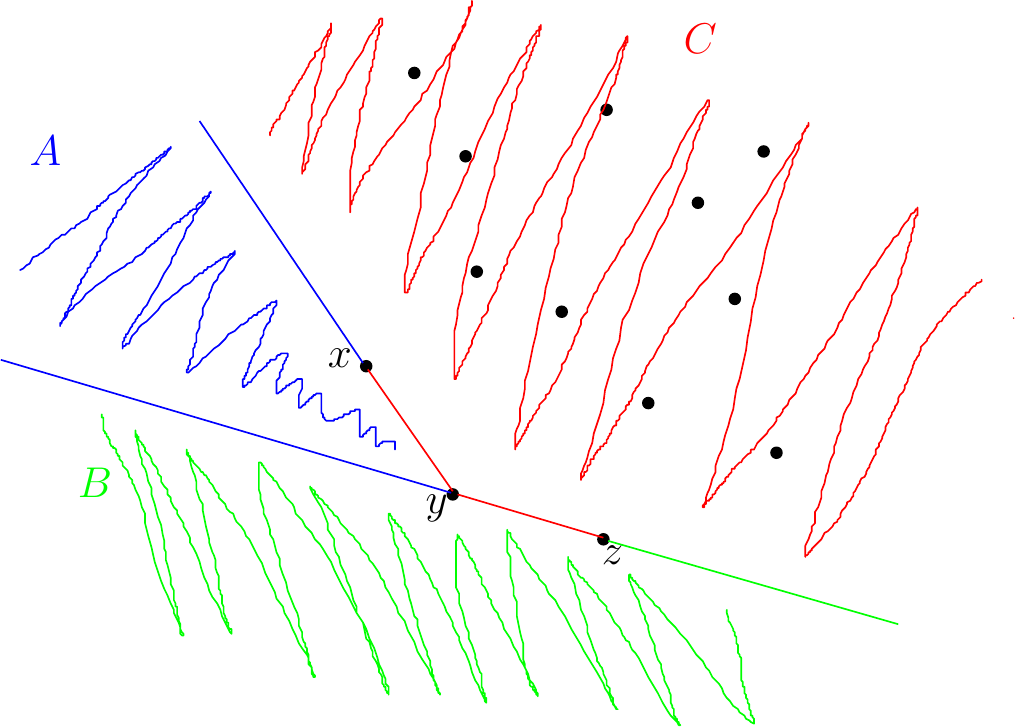}
		}
		\caption{An illustration for the proof of Proposition \ref{cl:g_1UB}. Each convex set is colored with a different color.}
		\label{fig:fig2}
	\end{center}
\end{figure}

	\bibliographystyle{plain}
	\bibliography{references}

\end{document}